\renewcommand{\H}{\mathcal{H}}
\newcommand{\fF}{f|_F}

\RequirePackage{fix-cm}

\documentclass[smallextended,envcountsect,envcountreset]{svjour3}        
\smartqed  
\usepackage{graphicx}
\usepackage{amssymb,amsmath,amsfonts,geometry}
\usepackage{color,url}

\begin{document}

\title{An enhanced Baillon-Haddad Theorem for Convex Functions defined on Convex Sets
}

\titlerunning{An enhanced Baillon-Haddad theorem for convex functions}        

\author{ Pedro P\'erez-Aros \and Emilio Vilches \thanks{E. Vilches was partially funded by CONICYT Chile under grant Fondecyt de Iniciaci\'on 11180098, P. P\'erez-Aros was partially supported by CONICYT Chile under grant Fondecyt regular 1190110.}
}


\institute{ 
	P. P\'erez-Aros \at
	Instituto de Ciencias de la Ingenier\'ia, Universidad de O'Higgins, Rancagua, Chile \\
	\email{pedro.perez@uoh.cl}           \and E. Vilches \at
	Instituto de Ciencias de la Educaci\'on \& Instituto de Ciencias de la Ingenier\'ia, Universidad de O'Higgins, Rancagua, Chile\\
	\email{emilio.vilches@uoh.cl}
}

\date{Received: date / Accepted: date}

\maketitle

\begin{abstract}
The Baillon-Haddad theorem establishes that the gradient of a convex and continuously differentiable function defined in a Hilbert space is $\beta$-Lipschitz if and only if it is $1/\beta$-cocoercive. In this paper, we extend this theorem to G\^{a}teaux differentiable convex functions defined on an open convex set of a Hilbert space. Finally, we give a characterization of  $C^{1,+}$ convex functions in terms of local cocoercitivity. 
 
\end{abstract}

\section{Introduction}
Let $\H$ be a Hilbert space endowed with a scalar product $\langle \cdot,\cdot \rangle$, induced norm $\Vert \cdot \Vert$ and unit ball $\mathbb{B}$. Given a nonempty {open convex} set $\Omega\subset \H$ and $\beta>0$, we say that an operator $T\colon \Omega \to \H$ is $1/\beta$-cocoercive if for all $x,y\in \Omega$

\begin{equation}\label{eq.1}
\beta \langle Tx-Ty,x-y\rangle \geq \Vert Tx-Ty\Vert^2,
\end{equation}
and $T$ is $\beta$-Lipschitz continuous if for all $x,y\in \Omega$
\begin{equation}\label{eq.2}
\Vert Tx-Ty\Vert \leq \beta \Vert x-y\Vert. 
\end{equation} 

If $\beta=1$, then \eqref{eq.1} means that $T$ is firmly nonexpansive and \eqref{eq.2} that $T$ is nonexpansive (see, e.g., \cite[Chapter~4]{BC2017}). It is clear that  \eqref{eq.1} implies  \eqref{eq.2}, while the converse, in general, is false even for monotone operators (take for example $T\colon \mathbb{R}^2\to \mathbb{R}^2$ given by $T(x,y)=(-y,x)$). Despite of this negative result, the Baillon-Haddad theorem (\cite[Corollaire~10]{Baillon1977}) states that if $T$ is the gradient of a convex function, then \eqref{eq.1} and \eqref{eq.2} are equivalent. The precise statement is the following:
\begin{theorem}[Baillon-Haddad]\label{BH}
	Let $f\colon \H\to \mathbb{R}$ be convex, Fr\'echet differentiable on $\H$, and such that $\nabla f$ is $\beta$-Lipschitz continuous for some $\beta>0$. Then $\nabla f$ is $1/\beta$-cocoercive.
\end{theorem}
This prominent result provides an important link between convex optimization and fixed-point iteration \cite{Byrne}. Moreover, it has many applications in optimization and numerical functional analysis (see, e.g.,  \cite{Briceno2010,BC2017,Combettes2004,Tseng1991,Zhu-Marcotte}).  An improved version of  Theorem \ref{BH} appeared in \cite{BC2010} (see also \cite[Theorem~1.2]{Byrne}), where the authors relate the Lipschitzianity of  the gradients of a convex function with the convexity and Moreau envelopes of associated functions (see \cite[Theorem~2.1]{BC2010}). Furthermore, they provided the following Baillon-Haddad theorem for twice continuously differentiable convex functions defined on open convex sets.
\begin{theorem}{\cite[Theorem~3.3]{BC2010}}\label{Teo3.3}
	Let $\Omega$ be a nonempty open convex subset of $\H$, let $f\colon \Omega \to \mathbb{R}$ be convex and twice continuously Fr\'echet differentiable on $\Omega$, and let $\beta>0$. Then $\nabla f$ is $\beta$-Lipschitz continuous if and only if it is $1/\beta$-cocoercive.
\end{theorem}
Finally, the authors left as an open question the validity of  Theorem \ref{Teo3.3} for simple G\^{a}teaux differentiable convex functions instead of twice continuously differentiable (see \cite[Remark~3.5]{BC2010}).

The aim of this paper is to extend  Theorem \ref{Teo3.3} to merely G\^{a}teaux differentiable convex functions (see Theorem \ref{infinite}). To do that, we first establish the result in finite-dimensions and then we use a finite dimensional reduction.

We emphasize that extend  Theorem \ref{Teo3.3} is of interest because it provides an important link between the gradient of convex functions defined on convex sets and cocoercive operators defined on convex sets.

Cocoercivity arises in various areas of optimization and nonlinear analysis  (see, e.g., \cite{Attouch2015,BC2017,Bot2017,Contreras2018}).
In particular, it plays an important role in the design of algorithms to solve structured monotone inclusions (which includes fixed points of non-expansive operators). Indeed, let us consider the structured monotone inclusion: find $x\in \H$ such that
\begin{equation}\label{inclusion-monotona}
0\in \partial \Phi(x)+Bx,
\end{equation}
where $\Phi\colon \H \to \mathbb{R}\cup \{+\infty\}$ is a {proper} convex lower semicontinuous function and $B\colon \H \to \H$ is a monotone operator. It is well known that (see, e.g., \cite{Attouch2015}) the problem \eqref{inclusion-monotona} is equivalent to the fixed point problem: find $x\in \H$ such that
\begin{equation}\label{fixed-point}
x=\operatorname{prox}_{\mu \Phi}\left(x-\mu Bx\right),
\end{equation}
where $\mu>0$ and $\operatorname{prox}_{\mu \Phi}\colon \H \to \H$ is the proximal mapping of $\Phi$  (see, e.g., \cite[Definition~12.23]{BC2017}) defined by 
$$
\operatorname{prox}_{\mu \Phi}(x):=\operatorname{argmin}_{y\in \H}\left\{ \Phi (y)+\frac{1}{2\mu} \Vert y-x\Vert^2\right\}. 
$$
To solve the fixed point problem \eqref{fixed-point},  Abbas and Attouch  \cite{Attouch2015} introduces the following dynamical system
\begin{equation}\label{Dyn}
\begin{aligned}
\dot{x}(t)&+x(t)=\operatorname{prox}_{\mu \Phi}\left(x(t)-\mu Bx(t)\right),\\
x(0)&=x_0,
\end{aligned}
\end{equation}
whose equilibrium points are solutions of \eqref{fixed-point}. They proved the following result (see \cite[Theorem~5.2]{Attouch2015})
\begin{proposition}
	Let $\Phi\colon \H \to \mathbb{R}\cup \{+\infty\}$ be a convex lower semicontinuous proper function, and $B$ a maximal monotone operator which is $\beta$-cocoercive {for some $\beta>0$}. Suppose that $\mu \in (0,2\beta)$ and $$\operatorname{zer}\left(\partial \Phi +B\right):=\{z\in \H \colon 0\in \partial \Phi(z)+Bz\}\neq \emptyset.$$  Then the unique solution of \eqref{Dyn} weakly converges to some element in  $\operatorname{zer}\left(\partial \Phi +B\right)$.
\end{proposition}
The previous result was extended by Bo\c{t} and Csetnek (see \cite[Theorem~12]{Bot2017}) to solve the monotone inclusion: find $x\in \H$ such that
\begin{equation}\label{maximal-inclusion}
0\in Ax+Bx,
\end{equation}
where $A\colon \H \rightrightarrows \H$ is a maximal monotone operator and $B\colon \H \to \H$ is $\beta$-cocoercive.  
It is important to emphasize that in order to solve the problems \eqref{fixed-point} and \eqref{maximal-inclusion}, it is enough that the operator $B$ is defined in $\operatorname{dom}\partial \Phi$ and $\operatorname{dom}A$, respectively. Therefore, it is interesting to have characterizations of cocoercive operators defined on open convex subsets of $\H$. Thus, it is important to extend  Theorem \ref{Teo3.3} to merely G\^{a}teaux differentiable functions (see \cite[Remark~3.5]{BC2010}).

\noindent The paper is organized as follows. After some preliminaries, in  Section \ref{mainresult} we state and prove the main result of the paper, that is, we prove the Baillon-Haddad theorem for G\^{a}teaux differentiable convex functions defined on open convex sets of arbitrary Hilbert spaces (see Theorem \ref{infinite}). Next, we give a characterization of $C^{1,+}$ convex functions in terms of local cocoercitivity (see Corollary \ref{CorollaryC1}). 
The paper ends with conclusions and final remarks.

\section{Notation and Preliminaries}


Given an open convex set $\Omega\subset \H$, we denote by $C^{1,+}(\Omega)$ the class of Fr\'echet differentiable functions $f\colon \Omega\subset \H \to \mathbb{R}$ whose gradient $\nabla f$ is locally Lipschitz (see, e.g., \cite[Chapter~9]{Rockafellar}).

\begin{example}\label{ejemplo}The following list provides some examples of cocoercive operators (we refer to \cite[Chapter~4]{BC2017} for further properties on cocoercive operators):
	\begin{enumerate}
		\item[(i)]   $T\colon \Omega\to \H$ is nonexpansive if and only if $I-T$ is $1/2$-cocoercive.
		\item[(ii)]\label{2T-char}  $T\colon \Omega\to \H$ is $1$-cocoercive if and only if $2T-I$  is $1$-Lipschitz.
		\item[(iii)] A matrix $M$ is psd-plus (that is, $M=E^t AE$ for $E$ any matrix and $A$ positive definite) if and only if the mapping $x\mapsto Mx$ is cocoercive (see \cite[Proposition~2.5]{Zhu-Marcotte}).
		\item[(iv)] The Yosida approximation $A_{\lambda}:=\frac{1}{\lambda}(\operatorname{Id}-(\operatorname{Id}+A)^{-1})$ of a maximal monotone operator $A\colon \H \rightrightarrows \H$ is $\lambda$-cocoercive (see \cite[Corollary~23.11]{BC2017}).
	\end{enumerate}
\end{example}
For a convex function $f\colon \Omega \subset \H \to \mathbb{R}$, we consider the convex subdifferential of $f$ at $x\in \Omega$ as 
{
$$
\partial f(x):=\{ x^*\in \H \colon f(x)+\langle x^*,y-x\rangle \leq  f(y) \textrm { for all } y\in \Omega \}.
$$}
It is well-known that for two functions $f, g\colon \Omega \subset \H \to \mathbb{R}$ the following equality holds (see, e.g., \cite[Corollary~16.48]{BC2017}):
\begin{equation}\label{suma}
\begin{aligned}
\partial (f+g)(x)&=\partial f(x)+\partial g(x) & \textrm{ for all } x\in \Omega.
\end{aligned}
\end{equation}

To prove our main result, we will use finite dimensional  reduction arguments, thus, some elements of generalized differentiation in finite dimensions will be needed. We refer to  \cite{Rockafellar} for more details.



\noindent Let  $f\colon \Omega\subset \mathbb{R}^n \to \mathbb{R}$ be a $C^{1,+}(\Omega)$ function. For $\bar{x}\in \Omega$, we define the Generalized Hessian of $f$ at $\bar{x}$ (see, e.g., \cite[Theorem~9.62]{Rockafellar} and \cite{HU1984}) as the set of matrices 
\begin{equation*}
\overline{\nabla}^2 f(\bar{x}):=\{A\in \mathbb{R}^{n\times n} \mid \exists x_n\to \bar{x}, x_n\in D, \nabla^2 f(x_n)\to A\},
\end{equation*}
where $D\subset \Omega$ is the dense set of points where $f$ is twice differentiable (by virtue of Rademacher's theorem the set $D$ exists).  
The following result (see \cite[Theorem~13.52]{Rockafellar}) establishes some properties of the Generalized Hessian $\overline{\nabla}^2f(\bar{x})$.
\begin{proposition}Let $f\colon \Omega\subset \mathbb{R}^n \to \mathbb{R}$ be a $C^{1,+}(\Omega)$ function, where $\Omega\subset \mathbb{R}^n$ is an open set. Then $\overline{\nabla}^2 f(\bar{x})$ is a nonempty, compact set of symmetric matrices.
\end{proposition}
The following result gives a known characterization of convexity  and Lipschitzianity of functions (see, e.g., \cite{Rockafellar,HU1984}). We give a proof for completeness.
\begin{proposition}\label{convex-char}
	Let $f\colon \Omega\to \mathbb{R}$ be a $C^{1,+}(\Omega)$ function with $\Omega\subset \mathbb{R}^n$ { open and }convex. Then 
	\begin{enumerate}
		\item[(i)] $f$ is convex if and only if for all $x\in \Omega$ and all $A\in \overline{\nabla}^2{f}(x)$ one has
		\begin{equation*}
		\left\langle Au,u\right\rangle \geq 0 \textrm{ for all } u\in \mathbb{R}^n.
		\end{equation*} 
		\item[(ii)] $\nabla f$ is $1$-Lipschitz on $\Omega$ if and only if for all $x\in \Omega$ and all $A\in \overline{\nabla}^2 {f}(x)$ the inequality $\Vert A\Vert \leq 1$ holds.
	\end{enumerate}
\end{proposition}
\begin{proof}
	\noindent $(i)$ follows from \cite[Example~2.2]{HU1984}. The necessity in $(ii)$ is direct. To prove the sufficiency in $(ii)$, assume that for all $x\in \Omega$ and all $A\in \overline{\nabla}^2 {f}(\bar{x})$ the inequality $\Vert A\Vert \leq 1$ holds. 
	Fix $y\in \mathbb{B}$ and consider the function $g_y(x):=\langle \nabla f(x),y\rangle$. Then, $g_y$ is locally Lipschitz on $\Omega$ {(because $f\in C^{1,+}(\Omega)$) and by virtue of Rademacher's theorem, the following formula holds:
	\begin{equation*}
	\begin{aligned}
	\nabla g_y(x)&=\langle \nabla^2 f(x),y\rangle & \textrm{ a.e. } x\in \Omega.
	\end{aligned}
	\end{equation*} }
	Thus, 
	\begin{equation*}
	\begin{aligned}
	\sup_{w \in \overline{\nabla } g_y(x)}\Vert w\Vert  &\leq  \Vert y\Vert  & \textrm{ for all } x\in \Omega,
	\end{aligned}
	\end{equation*}
	where $ \overline{\nabla } g_y(x)$ denotes the generalized gradient of $g_y$ (see \cite[Theorem~9,61]{Rockafellar}).	Hence, according to \cite[Theorem~3.5.2]{Mordukhovich2006}, the map $g_y$ is $\Vert y\Vert$-Lipschitz on $\Omega$. Finally, by virtue of \cite[Exercise~9.9]{Rockafellar}, we conclude that $\nabla f$ is $1$-Lipschitz on $\Omega$
\end{proof}

\section{An enhanced Baillon-Haddad theorem}\label{mainresult}

In this section, we state and prove the main result of the paper, that is, the Baillon-Haddad theorem for convex functions defined on convex sets, which extends \cite[Theorem~3.3]{BC2010} and solves the question posed in \cite[Remark~3.5]{BC2010}.
\begin{theorem}\label{infinite}
	Let $\Omega$ be a nonempty open convex subset of a Hilbert space $\H$, let $f\colon \Omega\to \mathbb{R}$ be a convex function and $\beta \in\,  ]0,+\infty[$.
	Then the following are equivalent.
	\begin{enumerate}
		\item[(a)] $f$ is G\^{a}teaux differentiable on $\Omega$ and $\nabla f$ is $\beta$-Lipschitz continuous on $\Omega$.
		\item[(b)] the map $x\mapsto \frac{\beta}{2}\Vert x\Vert^2-f(x)$ is convex on $\Omega$.
		\item[(c)] $f$ is G\^{a}teaux differentiable on $\Omega$ and $\nabla f$ is $1/\beta$-cocoercive.
	\end{enumerate}
	Moreover, if any of the above conditions holds, then $f\in C^{1,+}(\Omega)$.
\end{theorem}
 A straightforward consequence of the above result is given by the following characterization of being $C^{1,+}(\Omega)$. {
\begin{corollary}\label{CorollaryC1}
		Let $\Omega$ be a nonempty open convex subset of a Hilbert space $\H$ and let $f\colon \Omega\to \mathbb{R}$ be a convex function. Then, the function $f$ is $C^{1,+}(\Omega)$ if and only if $\nabla f$ is locally cocoercive, that is, for every $x\in  \Omega$ there exists a neighborhood $U$ of $x$ and a constant $\beta$ such that $\nabla f$ is $\beta$-cocoercive on $U$. 
	\end{corollary}}
Before presenting the proof of Theorem \ref{infinite}, we illustrate our results with the following example.
\begin{example}\label{Example31}
Let us consider the convex function  $f: (-4,4) \to \mathbb{R}$ defined by
	\begin{equation*}
	f(x)=\left\{
	\begin{array}{cl}
\frac{1}{8}x^{3/2} + \frac{4}{4-x}, & \text{ if } x\in [0,4),\\
\frac{1}{4-x} , & \text{ if } x\in (-4,0).
	\end{array} \right.
	\end{equation*}
\begin{figure}[h!!]
\centering\includegraphics[scale=0.8]{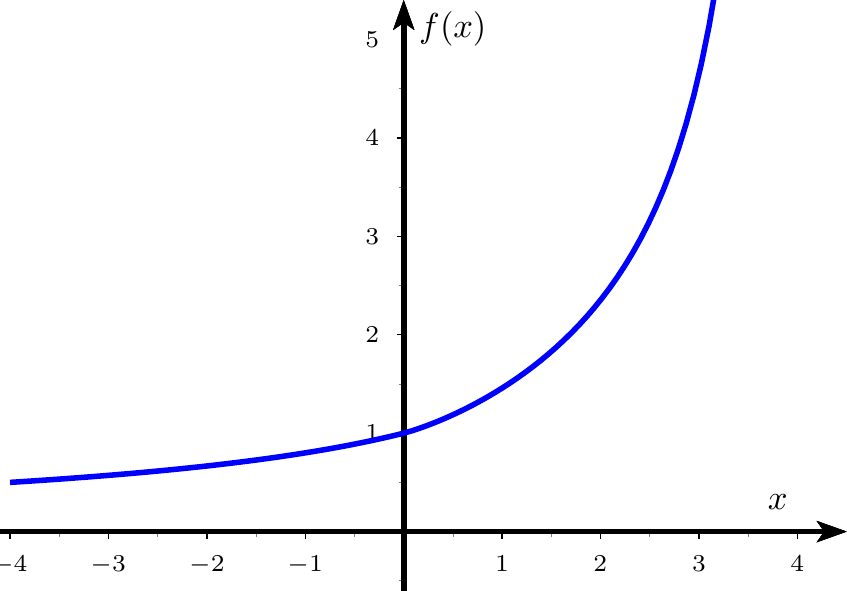}\caption{The function $f\colon (-4,4)\to \mathbb{R}$ from Example \ref{Example31}.}
\end{figure}	
\newline
It is clear that $f$ is $C^1$ over $\Omega = (-4,4)$, but not $C^2$. Indeed, 
\begin{align*}
	f'(x)=\left\{
	\begin{array}{cl}
\frac{3}{16}x^{1/2} + \frac{4}{(4-x)^2}, & \text{ if } x\in [0,4),\\
	\frac{4}{(4-x)^2} , & \text{ if } x\in (-4,0).
\end{array} \right.
\end{align*}
Moreover, $f$ cannot be extended to a differentiable function over the whole space $\mathbb{R}$. Consequently, the classical Baillon-Haddad theorem (Theorem \ref{BH}) and its extension for twice continuously differentiable functions (Theorem \ref{Teo3.3}) cannot be applied to $f$. However, by virtue of  Corollary \ref{CorollaryC1}, we know that the gradient of $f$ is locally cocoercive. Thus, due to Theorem \ref{infinite},  for every $\alpha  \in (0,4)$  the function $f$ is $\beta(\alpha)$-cocoercive on $(-\alpha, \alpha)$ with $\beta(\alpha)=1/f^{\prime}(\alpha)$.
\end{example}

To prove Theorem \ref{infinite}, we show first the result in finite dimension under the additional assumption that $f\in C^{1,+}(\Omega)$ (see the next lemma). Then, we obtain Theorem \ref{infinite} in finite dimensional spaces (see Lemma \ref{finite}). Finally, the proof of Theorem \ref{infinite} follows from finite dimensional reductions and Lemma \ref{finite}.
\begin{lemma}\label{lemma1}
	Let $\Omega$ be a nonempty open convex subset of $\mathbb{R}^n$, let $f\colon \Omega\to \mathbb{R}$ be a $C^{1,+}(\Omega)$ convex function and $\beta \in\, ]0,+\infty[$.
	Then the following are equivalent.
	\begin{enumerate}
		\item[(a)] $\nabla f$ is $\beta$-Lipschitz continuous on $\Omega$.
		\item[(b)] the map $x\mapsto \frac{\beta}{2}\Vert x\Vert^2-f(x)$ is convex on $\Omega$.
		\item[(c)] $\nabla f$ is $1/\beta$-cocoercive.
	\end{enumerate}
\end{lemma}
\begin{proof}
	Let us consider the functions 
	\begin{equation*}
	g(x):=\frac{1}{2}\Vert x\Vert^2-\frac{1}{\beta}f(x) \textrm{ and } h(x):=\frac{2}{\beta} f(x)-\frac{1}{2}\Vert x\Vert^2.
	\end{equation*}
	It is clear that {
	\begin{equation}\label{equivalencia-1}
	A\in  \overline{\nabla}^2 \left(\frac{f}{\beta}\right)(x) \Leftrightarrow B:=I-A\in \overline{\nabla}^2{g}(x).
	\end{equation}}
	On the one hand, {
	\begin{equation*}
	\begin{aligned}
	&\quad \,\,\,\, \nabla f \textrm{ is } \beta\textrm{-Lipschitz continuous}\\
	&\Leftrightarrow (\forall x\in \Omega)(\forall A\in   \overline{\nabla}^2 \left({f}/{\beta}\right)(x))\, \Vert A\Vert \leq 1 \hspace{3.6cm}  (\textrm{by Proposition  \ref{convex-char} (ii)})\\
	&\Leftrightarrow (\forall x\in \Omega)(\forall A\in   \overline{\nabla}^2 \left({f}/{\beta}\right)(x)) (\forall u\in \mathbb{R}^n)\, 0\leq \langle u,Au\rangle \leq \Vert u\Vert^2  \hspace{0.5cm} (\textrm{by Proposition \ref{convex-char} (i)})\\
	&\Leftrightarrow (\forall x\in \Omega)(\forall A\in \overline{\nabla}^2 \left({f}/{\beta}\right)(x)) (\forall u\in \mathbb{R}^n)\, 0\leq \Vert u\Vert^2-\langle u,Au\rangle   \\
	&\Leftrightarrow (\forall x\in \Omega)(\forall B\in \overline{\nabla}^2{g}(x)) (\forall u\in \mathbb{R}^n)\, 0\leq \langle u,Bu\rangle   \hspace{4.1cm}  (\textrm{by \eqref{equivalencia-1}}) \\
	&\Leftrightarrow g \textrm{ is convex}  \hspace{6.9cm}   (\textrm{ by Proposition  \ref{convex-char} (i)}),
	\end{aligned}
	\end{equation*}}
	which shows that $(a)$ is equivalent to $(b)$. \newline \noindent On the other hand,  {	\begin{equation*}
	\begin{aligned}
	&\quad \,\,\,\, g \textrm{ is convex},\\
	&\Leftrightarrow (\forall x\in \Omega)(\forall B\in \overline{\nabla}^2{g}(x)) (\forall u\in \mathbb{R}^n)\, 0\leq \langle u,Bu\rangle     \hspace{1.6cm} (\textrm{by Proposition  \ref{convex-char} (i)})\\
	&\Leftrightarrow (\forall x\in \Omega)(\forall A\in \overline{\nabla}^2 \left({f}/{\beta}\right)(x)) (\forall u\in \mathbb{R}^n)\, 0\leq \Vert u\Vert^2-\langle u,Au\rangle  \\
	&\Leftrightarrow (\forall x\in \Omega)(\forall A\in \overline{\nabla}^2 \left({f}/{\beta}\right)(x)) (\forall u\in \mathbb{R}^n)\, -\Vert u\Vert^2 \leq 2\langle u,Au\rangle -\Vert u\Vert^2 \leq \Vert u\Vert^2  & \\
	&\Leftrightarrow (\forall x\in \Omega)(\forall B\in \overline{\nabla}^2{h}(x))\, \Vert B\Vert \leq 1 \\
	&\Leftrightarrow  \textrm{ the map } x\mapsto \nabla h(x)=\frac{2}{\beta}\nabla f(x)-x \textrm{ is } 1\textrm{-Lipschitz} \hspace{0.6cm}  (\textrm{by Proposition \ref{convex-char} (ii)})\\
	&\Leftrightarrow  \nabla f \textrm{ is } 1/\beta\textrm{-cocoercive}  \hspace{5.6cm}  (\textrm{by Example \ref{ejemplo} (ii)}),
	\end{aligned}
	\end{equation*} }
	which proves that $(b)$ is equivalent to $(c)$.
\end{proof}
Now, we proceed to delete the hypothesis $f\in C^{1,+}(\Omega)$ from Lemma \ref{lemma1}. 
\begin{lemma}\label{finite}
	Let $\Omega$ be a nonempty open convex subset of $\mathbb{R}^n$, let $f\colon \Omega\to \mathbb{R}$ be a convex function and $\beta\in\, ]0,+\infty[$.
	Then the following are equivalent.
	\begin{enumerate}
		\item[(a)] $f$ is G\^{a}teaux differentiable on $\Omega$ and $\nabla f$ is $\beta$-Lipschitz continuous on $\Omega$.
		\item[(b)] the map $x\mapsto \frac{\beta}{2}\Vert x\Vert^2-f(x)$ is convex on $\Omega$.
		\item[(c)] $f$ is G\^{a}teaux differentiable on $\Omega$ and $\nabla f$ is $1/\beta$-cocoercive.
	\end{enumerate}
\end{lemma}
\begin{proof} 
	According to \cite[Theorem~2.2.1]{Borwein2010}, for {convex} functions defined on subsets of $\mathbb{R}^n$, G\^{a}teaux differentiability is equivalent to  Fr\'echet differentiablity. We proceed to show that any of the above conditions imply that $f\in C^{1,+}(\Omega)$. Indeed, it is clear that (a) and (c) implies that $f\in C^{1,+}(\Omega)$. To prove that (b) implies that $f\in C^{1,+}(\Omega)$, we follow some ideas from \cite{Byrne}. Let us define $h(x):= \frac{\beta}{2}\Vert x\Vert^2-f(x)${, for $x\in \Omega$}. Thus, 
	\begin{equation*}
	\begin{aligned}
	\frac{\beta}{2}\Vert x\Vert^2&=f(x)+h(x) & x\in \Omega,
	\end{aligned}
	\end{equation*}
	which implies that $\beta x=\partial f(x)+\partial h(x)$ for all $x\in \Omega$. Therefore, $\partial f(x)$ and $\partial h(x)$ are non-empty and contain a single element. Hence, by virtue of {\cite[Theorem~2.2.1]{Borwein2010}}, the function $f$ is G\^{a}teaux differentiable on $\Omega$ and, thus, Fr\'echet differentiable on $\Omega$ and continuously differentiable on $\Omega$ (see \cite[Theorem~2.2.2]{Borwein2010}).  It is not difficult to prove that (b) implies the following inequality: {
	\begin{equation*}
	\begin{aligned}
	\frac{\beta}{2}\Vert x-y\Vert^2 &\geq D_f(x,y):=f(x)-f(y)-\langle \nabla f(y),x-y\rangle \geq 0 & \textrm{ for all } x,y\in \Omega.
	\end{aligned}
	\end{equation*}
	Fix $y\in \Omega$ and define $d(x):=D_f(x,y)$. Then $\nabla d(x)=\nabla f(x)-\nabla f(y)$ and $D_f(z,x)=D_d(z,x)$ for all $z$ and $x$. Thus, we obtain 
	\begin{equation}\label{Bregman}
	\begin{aligned}
	\frac{\beta}{2}\Vert z-x\Vert^2 & \geq D_d(z,x)=d(z)-d(x)-\langle \nabla d(x),z-x\rangle & \textrm{ for all } z,x\in \Omega.
	\end{aligned}
	\end{equation} 
	Fix $\bar{x}\in \Omega$ and $\delta>0$ such that $\bar{x}+\delta \mathbb{B}\subset \Omega$  and $$M:=\sup_{z\in \bar{x}+\delta \mathbb{B}}\Vert \nabla f(z)\Vert<+\infty.$$ Let $x,y\in \bar{x}+\frac{\delta}{2} \mathbb{B}$ and $t\in (0,\min\{\frac{\delta}{4M},\frac{2}{\beta}\})$ such that $z:=x+t(\nabla f(y)-\nabla f(x))\in \bar{x}+\delta \mathbb{B}$.} Therefore, by taking $z=x+t(\nabla f(y)-\nabla f(x))$ in \eqref{Bregman} and using that $d(z)\geq 0$, we obtain  {
	\begin{equation*}
	D_f(x,y)\geq t\left(1-\frac{\beta t}{2}\right)\Vert \nabla f(x)-\nabla f(y)\Vert^2.
	\end{equation*}
	Analogously {(by taking $z:=y+t(\nabla f(x)-\nabla f(y))\in \bar{x}+\delta \mathbb{B}$)}, we get
	\begin{equation*}
	D_f(y,x)\geq t\left(1-\frac{\beta t}{2}\right)\Vert \nabla f(x)-\nabla f(y)\Vert^2.
	\end{equation*} }
	Thus, for all $x,y\in \bar{x}+\frac{\delta}{2}\mathbb{B}$
	\begin{equation*}
	\langle \nabla f(x)-\nabla f(y),x-y\rangle =D_f(x,y)+D_f(y,x) \geq t\left(2-\beta t\right)\Vert \nabla f(x)-\nabla f(y)\Vert^2,
	\end{equation*}
	which shows that $\nabla f$ is Lipschitz on $\bar{x}+\frac{\delta}{2}\mathbb{B}$. Therefore, $f\in C^{1,+}(\Omega)$
\end{proof}

\begin{remark} $D_f(x,y)$ is actually the Bregman distance between $x$ and $y$. Moreover, the equality $D_f(x,y)=D_d(z,x)$ is the so-called three points identity (see, e.g., \cite{Chen1993} or \cite[Lemma~3]{Bauschke2017-Bolte}).
\end{remark}

Now, we are ready to prove Theorem \ref{infinite}. \newline 
\textit{Proof of Theorem \ref{infinite}}\newline 
For a closed linear subspace $F\subset \H$, we denote by $\Vert \cdot\Vert_F$ the norm relative to $F$. We recall that 
$$
\Vert x\Vert_F=\sup_{h\in \mathbb{B}\cap F}\langle x,h\rangle, \text{ for all }x\in F.
$$

\noindent $(a) \Rightarrow  (b)$: Let $x,y\in \Omega$ and define $F:=\operatorname{span}\{x,y\}$. We observe that $\left(F,\langle \cdot,\cdot\rangle \right)$ is a finite dimensional Hilbert space. Thus the restriction of $f$ to $F$, $ \fF$, is G\^{a}teaux differentiable in $F$ and for all {$a,b\in \Omega \cap F$ and $h\in F$}
\begin{equation*}
\left\langle  \nabla \fF(a)-\nabla \fF(b),h \right\rangle =\left\langle  \nabla f(a)-\nabla f(b),h \right\rangle.
\end{equation*}
Hence, {for all $a,b\in \Omega \cap F$ and $h\in F$}
\begin{equation*}
\begin{aligned}
\Vert \nabla \fF(a)-\nabla \fF(b)\Vert_{F } &=\sup_{h\in \mathbb{B}\cap F}\left\langle \nabla \fF (a)-\nabla \fF (b),h\right\rangle \\
&\leq \Vert \nabla f(a)-\nabla f(b)\Vert_{\H  } \\
&\leq \beta \Vert a-b\Vert,
\end{aligned}
\end{equation*}
which shows that {$\nabla \fF$} is $\beta$-Lipschitz on $\Omega\cap F$. Therefore, according to Lemma \ref{finite}, the map $$x\mapsto h(x):=\frac{\beta}{2}\Vert x\Vert^2-\fF(x),$$ is convex on $\Omega\cap F$. In particular, for all $\lambda\in [0,1]$
$$
h(\lambda x+(1-\lambda)y)\leq \lambda h(x)+(1-\lambda)h(y).
$$
Since $x,y$ are arbitrary, it follows that the map $x\mapsto \frac{\beta}{2}\Vert x\Vert^2-f(x)$ is convex on $\Omega$.\newline
\noindent $(b) \Rightarrow  (a)$: We first observe that $x\mapsto h(x):=\frac{\beta}{2}\Vert x\Vert^2-f(x)$ is convex (with finite values) and for all $x\in \Omega$
$$
\frac{\beta}{2}\Vert x\Vert^2=f(x)+{h(x)}.
$$ 
Hence, by virtue of \eqref{suma}, for all $x\in \Omega$
$$
\beta x=\partial f(x)+\partial h(x),
$$
which implies that $\partial f(x)$ and $\partial h(x)$ are non-empty and contain a single element. Therefore, according to \cite[Corollary~4.2.5]{Borwein2010}, the function $f$ and $h$ are G\^{a}teaux differentiable on $\Omega$. Thus, if $F\subset \H$ is finite dimensional, then $h|_F$ is convex on {$\Omega\cap F$}. Hence, by virtue of Lemma \ref{finite}, {$\nabla \fF$} is $\beta$-Lispchitz on {$\Omega\cap F$, i.e., for all $x,y\in \Omega\cap F$}
\begin{equation}\label{asterisco}
\sup_{h\in \mathbb{B}\cap F}\langle \nabla f(x)-\nabla f(y),h \rangle =\Vert \nabla \fF(x)-\nabla \fF(y)\Vert_{F}\leq \beta \Vert x-y\Vert.
\end{equation} {
Let us consider
$$
\mathcal{F}_{x,y}:=\{ F \subset \H \colon F \textrm{ is a linear subspace of }\H \textrm{ with } x,y\in \Omega\cap F \textrm{ and } \operatorname{dim}F<+\infty \}.
$$}
Hence, since \eqref{asterisco} holds for any $F\subset \H$ finite dimensional, we obtain
\begin{equation}\label{asterisco2}
\begin{aligned}
\sup_{F\in \mathcal{F}_{x,y}} \Vert \nabla \fF(x)-\nabla \fF(y)\Vert_{F}& =\sup_{F\in \mathcal{F}_{x,y}} \sup_{h\in \mathbb{B}\cap F} \langle  \nabla f(x)-\nabla f(y),h\rangle \\
&=\Vert \nabla f(x)-\nabla f(y)\Vert_{\H}.
\end{aligned}
\end{equation}
Therefore, by taking supremum in \eqref{asterisco}, we conclude that for all $x,y\in \Omega$
$$
\Vert \nabla f(x)-\nabla f(y)\Vert \leq \beta \Vert x-y\Vert,
$$
which proves $(a)$. \newline 

\noindent $(c) \Rightarrow  (a)$: It is straightforward. \newline
{
\noindent $(a) \Rightarrow  (c)$: Let $x,y\in \Omega$ and $F\in \mathcal{F}_{x,y}$. Then $\left(F,\langle \cdot,\cdot\rangle \right)$ is a Hilbert space and   the restriction of $f$ to $F$, $\fF$, is G\^{a}teaux differentiable in $F$. Moreover,  
\begin{equation*}
\|  \nabla \fF(x)-\nabla \fF(y)\|_{F} \leq\|  \nabla f(x)-\nabla f(y)\|_{\H} \leq  \beta \| x- y\| \text{ for all }x,y\in \Omega\cap F.
\end{equation*} 
This implies that $\nabla \fF$ is $\beta$-Lipschitz on $ \Omega\cap F$. Whence,  by Lemma \ref{finite}, we have that $\nabla \fF $ is $1/\beta$-cocoercive on $ \Omega\cap F$. This implies in particular  that for $x,y\in \Omega$ and $F\in \mathcal{F}_{x,y}$ the following inequality holds
\begin{equation*}
\begin{aligned}
\beta \langle \nabla f(x)-\nabla f(y),x-y\rangle =\beta \langle \nabla \fF(x)-\nabla \fF(y),x-y\rangle 
&\geq \Vert \nabla f(x)-\nabla f(y)\Vert_{F}.
\end{aligned}
\end{equation*}
Since the above inequality is valid for all $F \in \mathcal{F}_{x,y}$, we can use again \eqref{asterisco2} to conclude that 
\begin{equation*}
\begin{aligned}
\beta \langle \nabla f(x)-\nabla f(y),x-y\rangle 
&\geq \Vert \nabla f(x)-\nabla f(y)\Vert_{\H}.
\end{aligned}
\end{equation*}

Hence, due to the fact that $x,y\in \Omega$ are arbitrary, we obtain that $\nabla f$ is $1/\beta$-cocoercive, which shows the equivalence between $(a)$, $(b)$ and $(c)$.

Moreover, it follows from $(a)$ (which is equivalent to $(b)$ and $(c)$) that $f$ is  G\^{a}teaux differentiable with $x\mapsto \nabla f(x)$ Lipchitz (and thus continuous). Hence, due to the \v{S}mulian's theorem (see, e.g., \cite[Theorem~4.2.10]{Borwein2010}), we get that $f$ is Fr\'echet differentiable on $\Omega$.  Therefore, any of the conditions $(a)$, $(b)$ and $(c)$ implies that $f\in C^{1,+}(\Omega)$, which ends the proof. \qed

}

\section{Conclusions  and final remarks}
In this paper, we have studied the Baillon-Haddad theorem for G\^{a}teaux differentiable convex functions defined on open convex sets of arbitrary Hilbert spaces. 
Our approach consists in the use of techniques from convex analysis, variational analysis, and nonsmooth analysis in conjunction with finite dimensional reductions. This paper improves known results in the literature and, in particular, gives a characterization of $C^{1,+}$ convex functions in terms of local cocoercitivity.
 We hope that the results of this paper shed light into the study of optimization algorithms that use local information of the objective function.

 \begin{acknowledgements}
The authors wish to thank the referees for providing several helpful suggestions. 
\end{acknowledgements}

\bibliographystyle{plain}
\bibliography{references}
\end{document}